\documentclass[12pt,twoside,a4paper]{article}  
\usepackage{amssymb,amsfonts,amsmath,amsbsy,theorem,amscd}

\newcommand{\ol}[1]{\overline{#1}}

\numberwithin{equation}{section}

\newcommand{\R}{\ensuremath{\mathbb{R}}}

\newcommand{\Rd}{\ensuremath{\mathbb{R}^d}}
\newcommand{\Rn}{\ensuremath{\mathbb{R}^n}}
\newcommand{\dm}{d}

\newcommand{\N}{\ensuremath{\mathbb{N}}}

\providecommand{\loc}{\operatorname{loc}}

\newcommand{\uloc}{\operatorname{uloc}}
\newcommand{\sd}{\, d}
\newcommand{\supp}{\operatorname{supp}}
\newcommand{\eps}{\ensuremath{\varepsilon}}
\newcommand{\weight}[1]{\langle #1\rangle}

\newcommand{\Div}{\operatorname{div}}

\newcommand{\sym}{\operatorname{sym}}

\newcommand{\Hzero}{H^{-1}_{(0)}}

\newcommand{\A}{\mathcal{A}}

\newcommand{\habil}[1]{}

\newtheorem{thm}{THEOREM}[section]

\newtheorem{lem}[thm]{Lemma}

\newtheorem{theorem}[thm]{Theorem}

\newtheorem{claim*}{Claim}
\newtheorem{assumption}[thm]{Assumption}
\theorembodyfont{\rmfamily}

\newenvironment{proof*}[1]{{\bf Proof #1:}}{\hspace*{\fill}\rule{1.2ex}{1.2ex}\\ } 
\newenvironment{proof}{{\bf Proof:\,}}{\hspace*{\fill}\rule{1.2ex}{1.2ex}\\ }

\newcommand{\ve}{\mathbf{v}} 
\newcommand{\ue}{\mathbf{u}} 
\newcommand{\tn}[1]{\mathbf{#1}}
\newcommand{\vc}[1]{\mathbf{#1}}

\newcommand{\Se}{\mathbf{S}}
\newcommand{\De}{\mathbf{D}}
\newcommand{\Ke}{\mathbf{K}}
\newcommand{\He}{\mathbf{H}}
\newcommand{\STC}{\varepsilon}
\newcommand{\varphie}{\boldsymbol{\varphi}}

\providecommand{\Qz}{\ensuremath{Q_0}}
\providecommand{\Bz}{\ensuremath{B_0}}
\providecommand{\Iz}{\ensuremath{I_0}}
  \providecommand{\abstmp}[2]{{#1\lvert{#2}#1\rvert}}
  \providecommand{\abs}[1]{\abstmp{}{#1}}

  \providecommand{\bfu}{{\bf u}} 
  \providecommand{\bfv}{{\bf v}}

  \providecommand{\bfD}{{\bf D}}

  \providecommand{\bfG}{{\bf G}}
  \providecommand{\bfH}{{\bf H}}

  \providecommand{\bfK}{{\bf K}}

  \providecommand{\bfS}{{\bf S}}

\begin{document}
\begin{titlepage}  \title{Existence of Weak Solutions for a Diffuse Interface Model of Non-Newtonian Two-Phase Flows} \author{ Helmut Abels\footnote{Fakult{\"a}t
      f{\"u}r Mathematik, Universit{\"a}t Regensburg, 93040 Regensburg,
      Germany, e-mail: {\sf
        helmut.abels@mathematik.uni-regensburg.de}}, Lars
    Diening\footnote{Mathematisches Institut, LMU M{\"u}nchen, 80333
      M{\"u}nchen, Germany: e-mail: {\sf diening@math.lmu.de}}, and
    Yutaka Terasawa\footnote{Graduate School of Mathematical Sciences, The University of Tokyo,
Tokyo, 153-8914, Japan: e-mail: yutaka@ms.u-tokyo.ac.jp}}
\end{titlepage}
\maketitle
\begin{abstract}
We consider a phase field model for the flow of two partly miscible incompressible, viscous fluids of Non-Newtonian (power law) type. In the model it is assumed that the densities of the fluids are equal. We prove existence of weak solutions for general initial data and arbitrarily large times with the aid of a parabolic Lipschitz truncation method, which preserves solenoidal velocity fields and was recently developed by Breit, Diening, and Schwarzacher.  
\end{abstract}
\noindent{\bf Key words:} Two-phase flow,  diffuse interface model, Non-Newtonian fluids,
 Cahn-Hilliard equation,  free boundary value problems, Lipschitz truncation

\noindent{\bf AMS-Classification:} 
Primary: 35Q35; 
Secondary:
35Q30, 
35R35,
76D05, 
76D45, 

\section{Introduction}

We consider the flow of two macroscopically immiscible, incompressible Non-Newtonian fluids. In contrast to classical sharp interface models, a partial mixing of the fluids is taken into account, which leads to a so-called diffuse interface model. This has the advantage that flows beyond the occurrence of topological singularities e.g. due to droplet collision or pinch-off can be described. More precisely we consider 
\begin{alignat}{2}\label{eq:NSCH1}
  \rho\partial_t \ve + \rho\ve\cdot \nabla \ve - \Div \tn{S}(c,\tn{D}\ve) + \nabla p &= -\STC\Div (\nabla c\otimes
\nabla c),
\\\label{eq:NSCH2}
  \Div \ve &=0, 
\\\label{eq:NSCH3}
  \partial_t c + \ve\cdot\nabla c &= m\Delta\mu, 
\\\label{eq:NSCH4}
\mu &= \STC^{-1}\phi(c) - \STC\Delta  c 
\end{alignat}
in $Q_T=\Omega\times (0,T)$, where $\Omega\subseteq \Rn$, $n\geq 2$, is a bounded domain and $T\in (0,\infty)$.
Here $\ve$ is the mean velocity, $\tn{D}\ve=\frac12(\nabla \ve+\nabla \ve^T)$, $p$ is the pressure, $c$ is an
order parameter related to the concentration of the fluids e.g. the concentration difference or the
concentration of one component, and $\rho$ is the density of the fluids, which is assumed to be constant. Moreover, $\tn{S}(c,\tn{D}\ve)$ is the viscous part of the stress tensor of the mixture to be specified below, $\STC>0$ is a (small)
parameter, which is related to the ``thickness'' of the interfacial region, 
 $\Phi\colon\R\to \R$ is a homogeneous free energy
density and $\phi=\Phi'$ and $\mu$ is the chemical potential. Capillary forces due to surface tension are modeled by an extra contribution $\STC \nabla c \otimes
\nabla c:= \STC \nabla c (\nabla c)^T$ in the stress tensor leading to the term on the right-hand side of (\ref{eq:NSCH1}). Moreover, we note
that in the modeling diffusion of the fluid components is taken into account. Therefore $m\Delta \mu$ is
appearing in (\ref{eq:NSCH3}), where $m>0$ is a constant mobility coefficient.

We close the system by adding the boundary and initial conditions
\begin{alignat}{2}\label{eq:NSCH5}
  \ve|_{\partial\Omega} &=0 &\qquad& \text{on}\ \partial\Omega\times (0,T),\\\label{eq:NSCH6}
  \vc{n}\cdot\nabla c|_{\partial\Omega} = \vc{n}\cdot\nabla \mu|_{\partial\Omega} &= 0
  &\qquad& \text{on}\ \partial\Omega\times (0,\infty),\\\label{eq:NSCH7}
  (\ve,c)|_{t=0} &= (\ve_0,c_0) &\qquad& \text{in}\ \Omega.
\end{alignat}
Here $\vc{n}$ denotes the exterior normal at $\partial\Omega$.
We note that (\ref{eq:NSCH1}) can be replaced by
\begin{equation}
  \label{eq:NSCH1b}
    \rho\partial_t \ve + \rho\ve\cdot \nabla \ve - \Div \tn{S}(c,\tn{D}\ve) + \nabla g = \mu\nabla c 
\end{equation}
with $g= p + \frac{\STC}2|\nabla c|^2+ \STC^{-1}\Phi(c)$ since 
\begin{equation}
  \label{eq:ExtraStressIdent}
\mu \nabla c= \nabla \left(\frac{\STC}2 |\nabla c|^2 +\STC^{-1}\Phi(c)\right) - \STC\Div (\nabla c\otimes \nabla c).   
\end{equation}

In the case of Newtonian fluids, i.e., $\tn{S}(c,\tn{D}\vc{v})= \nu(c)\tn{D}\ve$ for some positive viscosity coefficient $\nu(c)$, the model  was first discussed by Hohenberg and Halperin~\cite{HohenbergHalperin}. Later it was derived in the frame work of rational continuum mechanics by Gurtin,
Polignone, Vi\~{n}als \cite{GurtinTwoPhase}. The latter derivation can be easily modified to include a suitable non-Newtonian behavior of the fluids. If e.g. $\tn{S}(c,\tn{D}\ve)$ is chosen such that $\tn{S}(c,\tn{D}\ve): \tn{D}\ve\geq 0$, the local dissipation inequality, which yields thermodynamical consistency, remains valid.
For results on existence of weak and strong solutions in the case of Newtonian fluids we refer to Starovoitov~\cite{StarovoitovModelH}, Boyer~\cite{BoyerModelH}, and A.~\cite{ModelH}. First analytic results for the system \eqref{eq:NSCH1}-\eqref{eq:NSCH4} for Non-Newtonian fluids of power-law type were obtained by Kim, Consiglieri, and Rodrigues~\cite{NonNewtonianModelH}. The authors proved existence of weak solutions if $q\geq \frac{3d+2}{d+2}$, $d=2,3$, where $q$ is the power describing the growth of the stress tensor with respect to $\mathbf{D}\ve$. For this range of $q$ monotone operator techniques can be applied. Moreover, in the case $d=3$ and $2\leq q\leq \frac{11}5$ the authors prove existence of weak solutions. Grasselli and Pra\v{z}ak~\cite{GrasselliPrazakNonNewtonianDIM} discussed the longtime behavior of solutions of \eqref{eq:NSCH1}-\eqref{eq:NSCH4} in the case $q\geq \frac{3d+2}{d+2}$, $d=2,3$.

The goal of this article is to extend the existence result to lower values of $q$ in order to include the physically important case of shear thinning flows. In the case of a single fluid existence of weak solutions for power-law type fluids was proved for the case $q>\frac{2d}{d+2}$, $d\geq 2$, by D., R{\r u}{\v{z}}i{\v{c}}ka, and Wolf~\cite{DieningRuzickaWolf}. The proof is based on a parabolic Lipschitz truncation method and a careful decomposition of the pressure, which is needed since the  Lipschitz truncation used does not preserve the divergence freeness of a velocity field. Recently a parabolic Lipschitz truncation method, which keeps divergence free velocity fields divergence free, was developed by Breit, D. and Schwarzacher~\cite{BreitDieningSchwarzacher}. In the present article we will use this method in order to prove existence of weak solutions to \eqref{eq:NSCH1}-\eqref{eq:NSCH7} if $\tn{S}(c,\tn{D}\ve)$ is of power law type with an exponent $q>\frac{2d}{d+2}$. Precise assumptions are made in the following.

For simplicity we assume that $\STC=\rho=1$. But all results are true for general (fixed) $\STC,\rho>0$.
Moreover, we assume: 
\begin{assumption}\label{assump:1}
 Let $\Omega\subset\Rd$, $d=2,3$, be a bounded domain with $C^3$-boundary and let $\Phi\in C([a,b])\cap C^2((a,b))$ be such that $\phi=\Phi'$ satisfies 
\begin{equation*}
\lim_{s\to a}
\phi(s)=-\infty, \qquad \lim_{s\to b} \phi(s)=\infty, \qquad \phi'(s)\geq -\alpha  
\end{equation*}
for some $\alpha\in \R$. Let $m>0$ 
and let $\tn{S}\colon [a,b]\times \R^{d\times d}\to \R^{d\times d}$ be such that
\begin{eqnarray}
  |\tn{S}(c,\tn{M})|&\leq& C(|\sym(\tn{M})|^{q-1}+1)\\ 
  |\tn{S}(c_1,\tn{M})-\tn{S}(c_2,\tn{M})|&\leq& C|c_1-c_2|(|\sym(\tn{M})|^{q-1}+1)\\ 
  \tn{S}(c,\tn{M}):\tn{M}&\geq& \kappa |\operatorname{sym} (\tn{M})|^q -C_1
\end{eqnarray}
for all $\tn{M}\in \R^{d\times d}$, $c,c_1,c_2\in [a,b]$, and some $C,C_1,\kappa>0$, $q\in (\frac{2d}{d+2},\infty)$.
Moreover, we assume that $\tn{S}(c,\cdot)\colon \R^{d\times d}_{sym}\to \R^{d\times d}_{sym}$ is strictly monotone for every $c\in [a,b]$.
\end{assumption}

For the following we denote
\begin{equation*}
  E_{mix}(c)= \int_\Omega \frac{|\nabla c|^2}2 \sd x + \int_\Omega \varphi(c) \sd x.
\end{equation*}

Let $\ve\in L^q(0,T;W^1_{q,0}(\Omega)^d)\cap L^\infty(0,T;L^2_\sigma(\Omega))$, $c\in L^\infty(0,T;H^1(\Omega))\cap L^2(0,T;H^2(\Omega))$ with $\varphi (c)\in L^2(\Omega\times (0,T))$, and $\mu\in L^2(0,T;H^1(\Omega))$, where $0<T<\infty$. Then $(\ve,c,\mu)$ is a weak solution of the system \eqref{eq:NSCH1}-\eqref{eq:NSCH7} if 
for any $\varphie \in C^{\infty}(\ol{Q_T})^d $ with $ \Div \varphie = 0 $ and $ {\rm supp} (\varphi) \subset
\subset \Omega \times [0, T)$ the following holds:
\begin{alignat}{1}
\nonumber
-\int_{Q_T} \ve \cdot \partial_t \varphie ~d(x, t) &- \int_{Q_T} \ve \otimes \ve : \De \varphie~d(x, t)
+ \int_{Q_T} \Se(c, \De \ve) : \De \varphie~d(x, t)\\
&= \eps \int_{Q_T} \nabla c \otimes \nabla c : \De \varphie~d(x, t) 
+\int_{\Omega} \ve_0 \cdot \varphie(0) \,dx
\end{alignat}
and for every  $ \psi \in C^{\infty}(\overline{\Omega}\times [0,T]) $ with $\supp \psi \subset\subset \ol{\Omega}\times [0,T)$
\begin{alignat}{1}\nonumber
 -\int_{Q_T} c \partial_t \psi~d(x, t) - \int_{\Omega} c_0 \psi(0)\sd x 
&+\int_{Q_T} (\ve \cdot \nabla c) \psi~d(x, t)\\\label{eq:ConvCVWeak}
&= -m \int_{Q_T} \nabla \mu \cdot \nabla \psi~d(x, t), \\\nonumber
\mu &= \phi(c) - \Delta c, \\\nonumber
\vc{n} \cdot \nabla c|_{\partial \Omega} &= 0
\end{alignat}
holds.
\begin{thm}
  Let Assumption~\ref{assump:1} hold true and let $0<T<\infty$. Then for any $\ve_0\in L^2_\sigma(\Omega)$ and $c_0\in H^1(\Omega)$ with $c_0(x)\in [a,b]$ almost everywhere there exists a weak solution $\ve\in L^q(0,T;W^1_{q,0}(\Omega))\cap L^\infty(0,T;L^2_\sigma(\Omega))$, $c\in L^\infty(0,T;H^1(\Omega))\cap L^2(0,T;H^2(\Omega))$ with $\varphi (c)\in L^2(\Omega\times (0,T))$, and $\mu\in L^2(0,T;H^1(\Omega))$ in the sense above.
\end{thm}

The structure of the article is as follows. In Section~\ref{sec:Preliminaries} we summarize some preliminary results needed in the following analysis. In Section~\ref{sec:Approx} prove existence of solutions of an approximate system to \eqref{eq:NSCH1}-\eqref{eq:NSCH7}, where the convective terms $\ve\cdot \nabla \ve$, $\ve\cdot \nabla c$, and the capillary term $\Div(\nabla c\otimes \nabla c)$ are smoothed in a suitable way. Finally, in Section~\ref{sec:Main} the existence of weak solutions is proved by passing to the limit in the approximate system with the aid of a solenoidal parabolic Lipschitz truncation method.

\section{Preliminaries}\label{sec:Preliminaries}

Throughout the paper the usual Lebesgue spaces with respect to the Lebesgue measure are denoted by $L^p(M)$, $1\leq p\leq \infty$, for some measurable $M\subset\R^N$. Moreover, $L^p(M;X)$ denotes its Banachspace-valued variant and $L^p(0,T;X)=L^P((0,T);X)$. Furthermore, $f\colon [0,\infty)\to X$ is in $L^p_{\loc}([0,\infty);X)$ if $f$ is strongly measurable and $f|_{[0,T]}\in L^p(0,T;X)$ for every $0<T<\infty$ and
\begin{eqnarray*}
  L^p_{\uloc}([0,\infty);X)&:=& \left\{f\in L^p_{\loc}([0,\infty);X): \|f\|_{L^p_{\uloc}([0,\infty);X)}<\infty\right\}\\
\|f\|_{L^p_{\uloc}([0,\infty);X)} &:=& \sup_{T\geq 0} \|f\|_{L^p(T,T+1;X)}. 
\end{eqnarray*}
 The standard $L^p$-Sobolev space is denoted by $W^m_p(\Omega)$. $W^m_{p,0}(\Omega)$ is the closure of $C_0^\infty(\Omega)$ in $W^m_p(\Omega)$ and $H^m(\Omega)=W^m_2(\Omega), H^m_0(\Omega)=W^m_{2,0}(\Omega)$. Furthermore we use the  notation $L^2_{(0)}(\Omega)=\{f\in L^2(\Omega): \int_{\Omega} f(x) \sd x=0\}$, $H^1_{(0)}(\Omega)= H^1(\Omega)\cap L^2_{(0)}(\Omega)$, and $H^{-1}_{(0)}(\Omega):= H^1_{(0)}(\Omega)'$. Finally $L^2_\sigma(\Omega)$ is the closure of divergence free $C_0^\infty(\Omega)$-vector fields in $L^2(\Omega)^d$.

We recall some results on the Cahn-Hilliard equation with convection term:
\begin{alignat}{2}\label{eq:CH1}
  \partial_t c + \ve\cdot\nabla c &= m\Delta \mu &\qquad&\text{in}\ \Omega\times (0,\infty), \\\label{eq:CH2}
\mu &= \phi(c) - \Delta c&\qquad& \text{in}\ \Omega\times (0,\infty),\\\label{eq:CH3}
  \vc{n}\cdot\nabla c|_{\partial\Omega} = \vc{n}\cdot\nabla  \mu|_{\partial\Omega} &= 0
  &\qquad& \text{on}\ \partial\Omega\times (0,\infty),\\\label{eq:CH4}
  c|_{t=0} &= c_0 &\qquad& \text{in}\ \Omega
\end{alignat}
for given $c_0$ with $E_{mix}(c_0)<\infty$ and $\ve\in L^\infty(0,\infty;L^2_\sigma(\Omega))\cap L^2(0,\infty;H^1_0(\Omega)^d)$. Here $\phi=\Phi'$ and $\Phi$ is as in Assumption~\ref{assump:1}.
In the following \eqref{eq:CH1} together with \eqref{eq:CH3} will be understood in the following weak form
\begin{equation*}
  \weight{\partial_t c(t),\varphi}_{H^{-1}_{(0)},H^1_{(0)}} + \int_\Omega \ve(x,t)\cdot \nabla c(x,t)\varphi (x) \sd x =- m\int_\Omega \nabla \mu(x,t)\cdot \nabla \varphi(x)\sd x 
\end{equation*}
for all $\varphi \in H^1_{(0)}(\Omega)$ and almost every $t\in (0,T)$, where $\weight{\cdot,\cdot}_{X',X}$ denotes the duality product. Moreover, $Q_t:= \Omega\times (0,t)$, $Q=\Omega \times (0,\infty)$.

\begin{thm}\label{thm:ExistenceCH}
  Let $\ve\in L^2(0,\infty;H^1_0(\Omega)^d)\cap L^\infty(0,\infty;L^2_\sigma(\Omega))$. Then
  for every $c_0\in H^1_{(0)}(\Omega)$ with $E_{mix}(c_0)<\infty$ 
  there is a unique solution $c\in BC([0,\infty);H^1_{(0)}(\Omega))$ of (\ref{eq:CH1})-(\ref{eq:CH4}) 
with $\partial_t c \in L^2(0,\infty;\Hzero(\Omega))$ and $\mu \in L^2_{\uloc}([0,\infty);H^1(\Omega))$. This solution satisfies
  \begin{equation}\label{eq:E1Estim} 
    E_{mix}(c(t)) + \int_{Q_t} m|\nabla \mu|^2 \sd (x,\tau) = E_{mix}(c_0)-  \int_{Q_t} \ve\cdot \mu\nabla c \sd (x,\tau)
  \end{equation}
  for all $t\in [0,\infty)$ and
  \begin{eqnarray}\nonumber
\lefteqn{    \|c\|_{L^\infty(0,\infty;H^1)}^2+ \|\partial_t c\|_{L^2(0,\infty; \Hzero)}^2 + \|\nabla \mu\|_{L^2(Q)}^2}\\\label{eq:CHEstim1}&\leq& C\left(E_{mix}(c_0)+ \|\ve\|_{L^2(Q)}^2\right)
\\\nonumber
  \lefteqn{  \|c\|_{L^2_{\uloc}([0,\infty);W^2_r)}^2+ \|\phi(c)\|_{L^2_{\uloc}([0,\infty);L^r)}^2}\\\label{eq:CHEstim2} 
 &\leq&  C_r\left(E_{mix}(c_0)+ \|\ve\|_{L^2(Q)}^2\right)
  \end{eqnarray}
   where $r=6$ if $\dm=3$ and $1<r<\infty$ is arbitrary if $d=2$. Here $C, C_r$ are independent of $\ve,c_0$. Moreover, for every $R>0$ the solution 
  \begin{equation*}
    c\in Y:=L^2_{\loc}([0,\infty);W^2_r(\Omega))\cap H^1_{\loc}([0,\infty);\Hzero(\Omega))    
  \end{equation*}
  depends continuously on 
  \begin{equation*}
    (c_0,\ve)\in X:=H^1(\Omega)\times L^2_{\loc}([0,\infty);L^2_\sigma(\Omega)) \quad \text{with}\ E_{mix}(c_0)+\|\ve\|_{L^2(0,\infty;H^1)}\leq R   
  \end{equation*}
  with respect to the weak topology on $Y$ and the strong topology on $X$.
\end{thm}
The theorem is proved in \cite[Theorem 6]{ModelH} in the case $m=1$. The case $m>0$ can be reduced to the case $m=1$ by a simple scaling in time and a scaling of the homogeneous free energy density $f$.

We need the following theorem, which is a summary of Theorem~2.14 and
Corollary~2.15 of~\cite{BreitDieningSchwarzacher}.
\begin{theorem}
  \label{thm:sol_lip}
  Let $I_0$ be an open time interval, let $B_0$ be a ball in~$\R^d$,
  and let $Q_0 := I_0 \times B_0$.  Let $q,\sigma \in (1,\infty)$ with
  $q,q'>\sigma>1$, where $q'= \frac{q}{q-1}$. Let $\zeta\in
  C^\infty_0(\frac 16 \Qz)$ with $\chi_{\frac 18 \Qz}\leq\zeta\leq
  \chi_{\frac 16 \Qz}$. Let $\bfu_m$ and $\bfG_m$ satisfy $\partial_t
  \bfu_m = - \Div \bfG_m$ in the sense of
  distributions~$\mathcal{D}_{\Div}'( \Qz)$, where $\mathcal{D}_{\Div}=\{\boldsymbol{\varphi}\in C^\infty_0(\Qz)^d:\Div \varphi =0\}$.  Assume that
  $\bfu_m$ is a weak null sequence in $L^q(\Iz; W^{1,q}(\Bz))$, a
  strong null sequence in~$L^\sigma(\Qz)$ and bounded in $L^\infty(I_0,
  L^\sigma(B_0))$. Further assume that $\bfG_m = \bfG_{1,m} +
  \bfG_{2,m}$ such that $\bfG_{1,m}$ is a weak null sequence in
  $L^{q'}(\Qz)$ and $\bfG_{2,m}$ converges strongly to zero in
  $L^\sigma(\Qz)$. Then there exist a double sequence of open
  sets~$\mathcal{O}_{m,k}$, $k,m\in\N$, with $\limsup_{m\to\infty}\abs{{\cal O}_{m,k}} \leq c\, 2^{-k}
  2^{- q\, 2^k}$ for all $k\in\N$ such that for every $\bfK \in L^{q'}(\frac 16 Q_0)$
  \begin{align*}
    \limsup_{m\rightarrow\infty} \bigg|\int
    \big(\big(\bfG_{1,m}+\bfK):\nabla\bfu_m \big) \zeta
    \chi_{\mathcal{O}_{m,k}^\complement} \,d(x,t)\bigg| &\leq c\, 2^{-k/q}.
  \end{align*}
\end{theorem}
The proof of Theorem~\ref{thm:sol_lip} is based on a solenoidal
Lipschitz truncation. In particular, the set~${\cal O}_{m,k}$ are
level sets of suitable maximal operators defined by $\bfu_m$ and
$\bfG_m$. It is a refinement of the parabolic Lipschitz truncation of
\cite{KinLew02} and~\cite{DieningRuzickaWolf}. The advantage of
Theorem~\ref{thm:sol_lip} is that the pressure can be completely
avoided by using a solenoidal Lipschitz
truncation. See~\cite{BreDieFuc12} for a solenoidal Lipschitz
truncation in the stationary case.

\section{Approximate System}\label{sec:Approx}

In order to approximate \eqref{eq:NSCH1}-\eqref{eq:NSCH7}
we consider
\begin{alignat}{2}\nonumber
  \partial_t \ve + \Div (\Phi_\eps(\ve)\ve\otimes \ve ) &- \Div \tn{S}(c,\tn{D}\ve) + \nabla p\\\label{eq:NSCH1'}
 &= 
 -\Psi_\eps(\Div(\nabla c\otimes \nabla c))&\quad&\text{in}\ \Omega\times (0,T),
\\\label{eq:NSCH2'}
  \Div \ve &=0, &\quad&\text{in}\ \Omega\times (0,T),
\\\label{eq:NSCH3'}
  \partial_t c + (\Psi_\eps \ve)\cdot\nabla c &= m\Delta\mu, &\quad&\text{in}\ \Omega\times (0,T),
\\\label{eq:NSCH4'}
\mu &= \phi(c) - \Delta  c. &\quad& \text{in}\ \Omega\times (0,T)
\end{alignat}
together with \eqref{eq:NSCH5}-\eqref{eq:NSCH7}, where 
 $\Psi_\eps \vc{w}= P_\sigma (\psi_\eps \ast \vc{w})|_{\Omega}$,
$\psi_\eps(x)= \eps^{-d}\psi(x/\eps)$, $\eps>0$, is a usual smoothing kernel such that $\psi(-x)=\psi(x)$ for all $x\in\Rn$, $\vc{w}$ is extended by $0$ outside of $\Omega$, and $P_\sigma$ is the Helmholtz projection. Moreover, $\Phi_\eps(\vc{s})=\Phi(\eps |\vc{s}|^2)$ for all $\vc{s}\in\R^d$, $\eps>0$ with some $\Phi\in C_0^\infty(\R)$ with $\Phi(0)=1$.

The approximate system is formulated weakly as follows:
For any $ \boldsymbol{\varphie} \in C^{\infty}(Q_T)^d $ with $ \Div \boldsymbol{\varphie} = 0 $ and $ {\rm supp} (\varphie) \subset
\subset \Omega \times [0, T),$
\begin{alignat}{1}
\nonumber 
-\int_{Q_T} \ve \cdot \partial_t \boldsymbol{\varphie} ~d(x, t) + \int_{Q_T}\left(\Se (c, \De \ve) - \ve \otimes \ve\,
\Phi_{\eps}(\ve)\right) : \De \boldsymbol{\varphie} ~d(x, t)\\\nonumber
=-\int_{Q_T} \Psi_{\eps}\left(\Div(\nabla c \otimes \nabla c)\right) \cdot \boldsymbol{\varphie} ~d(x, t)
+\int_{\Omega} \ve_0 \cdot \boldsymbol{\varphie}(0) dx, 
\end{alignat} 
holds and for any  $ \psi \in C^{\infty}( \overline{\Omega}\times [0, T)) $ with $ {\rm supp} (\varphie) \subset \subset  \overline{\Omega}\times[0, T) , $
\begin{alignat}{1}\nonumber
- \int_{Q_T} c\partial_t \psi ~d(x, t) &- \int_{\Omega} c_0 \psi(0)~dx
+ \int_{Q_T} (\Psi_\eps\ve\cdot\nabla c) \psi~d(x, t) \\\label{eq:WeakApprox2}
&= -\int_{Q_T} m\nabla \mu \cdot \nabla \psi~d(x, t)\\
\mathbf{n}\cdot \nabla c|_{\partial\Omega} &= 0
\end{alignat}
holds 
and
\begin{equation*}
  \mu = \varphi(c)-\Delta c \qquad \text{in}\ Q_T.
\end{equation*}

For the following let $V_p(\Omega)= W^1_{p,0}(\Omega)^d\cap L^p_\sigma(\Omega)$.
\begin{thm}\label{thm:ExistenceApprox}
Let $\frac{2d}{d+2}<q<\infty$, $d=2,3$.  For every $0<T<\infty$, $\ve_0\in L^2_\sigma(\Omega)$, $c_0\in H^1(\Omega)$ such that $c_0(x)\in [a,b]$ almost everywhere there is a weak solution $(\ve,c,\mu)$ of \eqref{eq:NSCH1'}-\eqref{eq:NSCH4'},\eqref{eq:NSCH5}-\eqref{eq:NSCH7} such that
  \begin{alignat*}{1}
    \ve&\in W^1_{q'}(0,T;V_q(\Omega)')\cap L^q(0,T;V_q(\Omega)),\\
    c&\in C([0,T];H^1(\Omega))\cap H^1(0,T;H^{-1}_{(0)}(\Omega))\cap L^2(0,T;W^2_r(\Omega)),\\
    \mu&\in L^2(0,T;H^1(\Omega))
  \end{alignat*}
  where $r=6$ if $d=3$ and $1\leq r<\infty$ is arbitrary if $d=2$. Moreover, for every $0\leq t\leq T$
  \begin{eqnarray}\nonumber
    \lefteqn{\frac12 \|\ve(t)\|_{L^2(\Omega)}^2+E_{mix}(c(t)) + \int_0^t\int_\Omega \tn{S}(c,\tn{D}\ve):\tn{D}\ve \sd x\sd \tau}\\\label{eq:EnergyEstimCH}
&& + \int_0^t\int_\Omega m|\nabla \mu|^2 \sd x\sd \tau=   \frac12 \|\ve_0\|_{L^2(\Omega)}^2+E_{mix}(c_0)=:E_0
  \end{eqnarray}
  and
  \begin{equation}\label{eq:EstimCH}
    \|c\|_{L^2(0,T;W^2_r(\Omega))}+\|\phi(c)\|_{L^2(0,T;L^r(\Omega))}\leq C(T,E_0)
  \end{equation}
  for some $C(T,E_0)>0$ depending continuously on $T,E_0$.
\end{thm}
\begin{proof}
  Let 
  \begin{equation*}
    X_1:= L^q(0,T;V_q(\Omega))\cap W^1_q(0,T;V_q(\Omega)'),\quad X_0 := L^2(0,T;L^2_\sigma(\Omega)),
  \end{equation*}
 and let $X_{\frac12}:=(X_0,X_1)_{[\frac12]}$. Then $  X_1\hookrightarrow \hookrightarrow X_0$ by the Lemma of Aubin-Lions,  cf. e.g. J.-L.~Lions~\cite{QuelquesMethodes} or Simon~\cite{SimonCompactSets}, and since $X_1\hookrightarrow L^\infty(0,T;L^2_\sigma(\Omega))$. Therefore
\begin{equation*}
  X_1\hookrightarrow \hookrightarrow X_{\frac12}\hookrightarrow \hookrightarrow X_0
\end{equation*}
due to \cite[Theorem 3.8.1]{Interpolation}.

  We define a mapping $F\colon X_{\frac12} \to X_\frac12$ as follows: Given $\ue\in X_{\frac12}$, let $c$ be the solution of (\ref{eq:CH1})-(\ref{eq:CH4}) due to Theorem~\ref{thm:ExistenceCH} with $\ve(x,t)= (\Psi_\eps\ue)(x,t)\chi_{[0,T]}(t)$ and $c_0$ as in the assumptions. Then $\ue\mapsto c$ is continuous from the strong topology of $X_0$ to the weak topology of
  \begin{equation*}
    Y=L^2(0,T;W^2_r(\Omega))\cap H^1(0,T;H^{-1}_{(0)}(\Omega))
  \end{equation*}
  as stated in Theorem~\ref{thm:ExistenceCH}. Therefore $\ue\mapsto c$ is weakly continuous from $X_{\frac12}$ to $Y$.
  Moreover, $X_0\ni \ue\mapsto c\in Y\cap L^\infty(0,T;H^1(\Omega))$ is a bounded mapping and 
  $
    Y\hookrightarrow \hookrightarrow L^2(0,T;C^1(\ol{\Omega}))
  $
  by the Lemma of Aubin-Lions. 
  Interpolation implies  that $Y\ni c\mapsto \nabla c \in L^6(0,T;L^2(\Omega))$ and 
  $Y\ni c\mapsto \nabla c\otimes \nabla c\in L^{q'}(0,T;L^1(\Omega))$ are completely continuous mappings.
Hence $X_{\frac12}\ni \ue \mapsto \nabla c\otimes \nabla c\in L^{q'}(0,T;L^1(\Omega))$ is completely continuous since $X_{\frac12}\ni \ue \mapsto c\in Y$ is weakly continuous.
 
 Now let 
  \begin{eqnarray*}
    \lefteqn{\weight{\mathbf{f},\varphi}=\weight{\mathbf{f}(\ue),\boldsymbol{\varphi}}}\\
&=& \int_{Q_T}(\nabla c\otimes \nabla c):\tn{D}(\Psi_\eps \boldsymbol{\varphi}) \sd (x,t) - \int_{Q_T}(\Phi_\eps(\ue) \ue\otimes \ue):\tn{D} \boldsymbol{\varphi} \sd (x,t)
  \end{eqnarray*}
for all $\boldsymbol{\varphi}\in L^q(0,T;V_q(\Omega))$ and let 
$\ve=F(\ue)\in X_1$ be the solution of the abstract evolution equation 
\begin{alignat}{2}\label{eq:Stokes1}
  \frac{d}{dt} \ve(t) + \mathcal{A}_c(t)\ve(t) &= \mathbf{f}(t)&\qquad& \text{in}\ L^{q'}(0,T;V_q(\Omega)'),\\\label{eq:Stokes2}
  \ve(0)&= \ve_0&\qquad& \text{in}\ L^2_\sigma(\Omega),
\end{alignat}
where
\begin{equation*}
  \weight{\mathcal{A}_c(t)(\ve(t)),\boldsymbol{\varphi}}_{V_q(\Omega)',V_q(\Omega)} = \int_\Omega \tn{S}(c(t),\tn{D}\ve(t)): \tn{D}\boldsymbol{\varphi}\sd x
\end{equation*}
for all $\boldsymbol{\varphi}\in V_q(\Omega)$. Then $\mathcal{A}_c(t)\colon V_q(\Omega)\to V_q(\Omega)'$ is a strictly monotone, bounded, hemi-continuous, and coercive  operator and existence of a unique solution follows from standard results on evolution equations for monotone operators, cf. e.g. \cite[Proposition~4.1]{Showalter}. Moreover, from standard estimates it follows that  the mapping
\begin{equation*}
  L^{q'}(0,T;V_q(\Omega)')\times Y\ni (\mathbf{f},c) \mapsto \ve \in L^q(0,T;V_q(\Omega))\cap W^1_q(0,T;V_q(\Omega)')\in X_1
\end{equation*}
is bounded. From this and the uniqueness of the solution one can derive that the latter mapping is also weakly continuous as follows: If
\begin{eqnarray*}
  (\mathbf{f}_k,c_k)\rightharpoonup_{k\to\infty} (\mathbf{f},c) \qquad \text{in}\ L^{q'}(0,T;V_q(\Omega)')\times Y,
\end{eqnarray*}
then $c_k\to_{k\to \infty} c$ in $L^2(0,T;C^1(\ol{\Omega}))$ by compact embedding. Now, if $\ve_k$ is the solution of the evolution equation above with $\mathbf{f}_k$ instead of $\mathbf{f}$. Then $(\ve_k)_{k\in\N}$ is bounded in $X_1$. Moreover, for any weakly convergent subsequence $(\ve_{k_j})_{j\in\N}$ we have that 
\begin{eqnarray*}
  \A_{c_{k_j}} (\ve_{k_j}) \rightharpoonup_{j\to \infty } \A_c (\ve)\qquad \text{in}\ L^{q'}(0,T;V_q(\Omega)').
\end{eqnarray*}
Hence $\ve_{k_j}\rightharpoonup_{j\to \infty} \ve\in X_1$, where $\ve$ is the unique solution of \eqref{eq:Stokes1}-\eqref{eq:Stokes2}. Because of uniqueness of the solution, this holds true for any weakly convergent subsequence. Therefore $\ve_k\rightharpoonup_{k\to \infty} \ve\in X_1$. 

Moreover, it is easy to prove that $X_{\frac12}\ni \ue \mapsto \mathbf{f}(\ue)\in X_0$ is completely continuous.
Altogether we obtain that the mapping $X_{\frac12}\ni \ue \mapsto \ve=F(\ue)\in X_1$ is weakly continuous. Therefore $X_{\frac12}\ni \ue \mapsto \ve=F(\ue)\in X_{\frac12}$ is completely continuous.

 In order to apply the Leray-Schauder principle to $F$,
  cf. e.g. \cite[Chapter~II, Lemma~3.1.1]{BuchSohr}, it only remains to show
  that there is some $R>0$ such that 
  \begin{equation*}
    \lambda F(\ue)= \ue \ \text{for some}\ \ue\in X_{\frac12}, \lambda\in [0,1]\quad \Rightarrow
    \quad \|\ue\|_{X_\frac12}\leq R.    
  \end{equation*} 
  Assume that $\lambda F(\ue)=\ue$ for some $\ue\in X_{\frac12}$, $\lambda \in (0,1]$. (The case $\lambda=0$ is trivial).
  Hence $\ve=\lambda^{-1} \ue$ solves (\ref{eq:Stokes1})-(\ref{eq:Stokes2}) with
  right-hand side $f(\ue)$ as above. Thus taking the  product of
  (\ref{eq:Stokes1}) and $\ve$  we conclude that
\begin{eqnarray*}
  \lefteqn{\frac12\|\ve(T)\|_2^2 + \int_{Q_T}\tn{S}(c,\tn{D}\ve):\tn{D}\ve \sd (x,t)}\\
  &=& \frac12 \|\ve_0\|_2^2 - \lambda^{-1}(\Psi_\eps \ue\cdot \nabla \ue, \ue)_{Q_T} -(\Div\nabla c\otimes \nabla c, \Psi_\eps \ve)_{Q_T}\\
&=& \frac12 \|\ve_0\|_2^2 +\lambda^{-1}(\mu \nabla c, \Psi_\eps\ue)_{Q_T}
\end{eqnarray*}
where we have used (\ref{eq:ExtraStressIdent}) with $\eps=1$. Combining this with
(\ref{eq:E1Estim}) we obtain
\begin{eqnarray*}
  \lefteqn{\frac12\|\ve(T)\|_2^2 + \int_{Q_T}\tn{S}(c,\tn{D}\ve):\tn{D}\ve\sd (x,\tau)}\\
&&+\frac1{\lambda} E_{mix}(c(T))+
  \frac1{\lambda}\int_{Q_T}m|\nabla\mu|^2\sd (x,\tau) 
  = \frac12 \|\ve_0\|_2^2 +\frac1{\lambda} E_{mix}(c_0)
\end{eqnarray*}
and therefore 
\begin{eqnarray*}
\lefteqn{\|\ue\|_{L^\infty(0,T;L^2(\Omega))}^2+\|\ue\|_{L^q(0,T;V_q)}^q}\\
&=& \lambda^2 \|\ve\|_{L^\infty(0,T;L^2)}^2 + \lambda^q\|\ue\|_{L^q(0,T;V_q)}^q\leq
CE(\ve_0,c_0).  
\end{eqnarray*}
  Because of
 (\ref{eq:CHEstim1}), (\ref{eq:CHEstim2}),
there is some
$R>0$ such that 
\begin{equation*}
\|\ue\|_{X_1}\leq M(\|\textbf{f}(\ue)\|_{L^{q'}(0,T;V_q')})\leq M' (\|\ue\|_{L^\infty(0,T;L^2)})\leq R,  
\end{equation*}
 where $M,M'\colon [0,\infty)\to [0,\infty)$ are continuous and non-decreasing functions. Hence we can apply the
Leray-Schauder principle to conclude the existence of a fixed point $\ve=F(\ve)$,
$\ve\in X_{\frac12}$. Since the solution of \eqref{eq:Stokes1}-\eqref{eq:Stokes2} is in $X_1$, we even have $\ve\in X_1$. Finally, the energy identity is proved by the same calculations as
above with $\lambda=1$ and $T$ replaced by $t\in (0, T)$.  Finally it is easy to observe that $(\ve,c,\mu)$, is a weak solution in the sense above, where $(c,\mu)$ are determined by (\ref{eq:CH1})-(\ref{eq:CH4}) with $\ve(x,t)= (\Psi_\eps\ue)(x,t)\chi_{[0,T]}(t)$.
\end{proof}
\begin{lem}
  Let $(\eps_j)_{j\in\N}$ be a null sequence and $(\ve_j,c_j,\mu_j)$ be the solutions of \eqref{eq:NSCH1'}-\eqref{eq:NSCH4'},\eqref{eq:NSCH5}-\eqref{eq:NSCH7} above with $\eps$ replaced by $\eps_j$. Assume that $\ve_j\to_{j\to \infty} \ve$ in $L^2(0,T;L^2_\sigma(\Omega))$. Then for a suitable subsequence
  \begin{alignat}{2}\label{eq:cjConv}
    c_j&\rightharpoonup_{j\to \infty} c &\quad & \text{in}\ L^2(0,T;W^2_r(\Omega)),\\\label{eq:StrongcjConv}
    c_j&\to_{j\to \infty} c &\quad & \text{in}\ L^4(0,T;W^1_4(\Omega)),\\\label{eq:mujConv}
    \mu_j&\rightharpoonup_{j\to \infty}\mu &\quad &  \text{in}\ L^2(0,T;H^1(\Omega))
  \end{alignat}
  where $(c,\mu)$ solve \eqref{eq:NSCH3}, $c|_{t=0}=c_0$, and $\vc{n}\cdot\nabla \mu|_{\partial\Omega}=0$ in the sense that \eqref{eq:ConvCVWeak} holds for any $ \psi \in C^{\infty}(\overline{\Omega}\times [0,T]) $ with $\supp \psi \subset\subset \ol{\Omega}\times [0,T)$. Moreover, \eqref{eq:NSCH4}  holds pointwise almost everywhere in $Q_T$ and $\vc{n}\cdot\nabla c|_{\partial\Omega}$ in $\partial\Omega\times (0,T)$ in the trace sense.
\end{lem}
\begin{proof}
Because of \eqref{eq:EnergyEstimCH}, \eqref{eq:EstimCH}, and \eqref{eq:NSCH4'}, $(c_j)_{j\in\N}$ and $(\mu_j)_{j\in\N}$ are bounded in $L^2(0,T;W^2_r(\Omega))$, $L^2(0,T,H^1(\Omega))$, respectively. Hence \eqref{eq:cjConv} and \eqref{eq:mujConv} hold for a suitable subsequence and some $c\in L^2(0,T;W^2_r(\Omega))$, $\mu\in L^2(0,T;H^1(\Omega))$. {} Moreover, \eqref{eq:NSCH3'} implies that $\partial_t c_j\in L^2(0,T;H^{-1}(\Omega))$ is bounded. Hence
\begin{equation*}
  c_j\to_{j\to\infty} c\qquad \text{in}\ L^2(0,T;C^1(\ol{\Omega}))
\end{equation*}
by the Lemma of Aubin-Lions and $W^2_r(\Omega)\hookrightarrow \hookrightarrow C^1(\ol{\Omega})$. Using that $(c_j)_{j\in \in\N}\subset BUC([0,T];H^1(\Omega))$ is bounded, a simple interpolation arguments yields \eqref{eq:StrongcjConv}.

In order to prove \eqref{eq:NSCH4}, we use a monotonicity argument. To this end let $\phi_0(s):= \phi(s)+\alpha s$ for all $s\in\R$ and let $\mathcal{A}\colon \mathcal{D}(A)\subseteq L^2(\Omega)\to L^2(\Omega)$ be defined $\mathcal{A}(c)= -\Delta c+\phi_0(c)$ for all $c\in \mathcal{D}(A)$ with
\begin{equation*}
  \mathcal{D}(\mathcal{A})= \{ u \in H^2(\Omega): \phi_0(u) \in L^2(\Omega),
            \phi_0'(u) |\nabla u|^2 \in L^1(\Omega), \partial_{\vc{n}} u |_{\partial \Omega} = 0 \}
\end{equation*}
Then $\phi_0\colon (a,b)\to \R$ is monotone, $\mathcal{A}$ is a maximal monotone operator,
and there is some $C>0$ such that
\begin{equation}\label{eq:AEstim}
  \|c\|_{H^2(\Omega)} + \|\phi_0(c)\|_{L^2(\Omega)}\leq C\left(\|\mathcal{A}(c)\|_{L^2(\Omega)}+ \|c\|_{L^2(\Omega)}+ 1\right)
\end{equation}
for all $c\in\mathcal{D}(A)$
 because of  \cite[Theorem 3.12.8]{Habilitation}, which is a variant of \cite[Theorem 4.3]{AsymptoticCH}.

Therefore $\mathcal{A}_T\colon\mathcal{D}(\mathcal{A}_T)\subset L^2(\Omega\times(0,T))\to L^2(\Omega\times (0,T))$ defined by $(\mathcal{A}_T c)(t)= \mathcal{A} (c(t))$ for almost every $t\in (0,T)$ and $c$ in
\begin{eqnarray*}
  \mathcal{D}(\mathcal{A}_T)&=& \left\{u\in L^2(\Omega\times (0,T)): u(t)\in \mathcal{D}(\mathcal{A})\ \text{for a.e.}\ t\in (0,T),\right.\\
&&\left. \ \mathcal{A}(u(\cdot))\in L^2(\Omega\times (0,T)\right\}
\end{eqnarray*}
is a monotone operator. Since $\mathcal{A}$ is maximal monotone, $I+\mathcal{A}\colon \mathcal{D}(A)\to L^2(\Omega)$ is bijective. Combining this with \eqref{eq:AEstim}, one easily obtains that also $I+\mathcal{A}_T\colon \mathcal{D}(A_T)\to L^2(\Omega\times (0,T))$ is bijective, cf. e.g. \cite[Lemma 1.3, Chapter IV]{Showalter}.
Moreover, because of \eqref{eq:NSCH4'},
\begin{equation*}
\mathcal{A}_T(c_j)=  \phi_0(c_j) -\Delta c_j\rightharpoonup_{j\to\infty} \mu-\alpha c \qquad \text{in}\ L^2(\Omega\times (0,T)).
\end{equation*}
Furthermore, 
\begin{equation*}
  \int_0^T\int_\Omega \left(\mu_j-\alpha c_j \right) c_j \sd x\sd t\to_{j\to\infty}   \int_0^T\int_\Omega \left(\mu-\alpha c \right) c \sd x\sd t
\end{equation*}
since $c_j\to_{j\to\infty} c$ and $\mu_j\rightharpoonup_{j\to\infty} \mu$ in $L^2(\Omega\times (0,T))$. Hence
\begin{equation*}
  \int_0^T\int_\Omega \mathcal{A}(c_j) c_j \sd x\sd t \to_{j\to\infty} \int_0^T\int_\Omega (\mu-\alpha c) c \sd x\sd t
\end{equation*}
and \cite[Proposition IV.1.6]{Showalter} implies that $c\in \mathcal{D}(\mathcal{A}_T)$ and 
\begin{equation*}
  \mathcal{A}_T (c) = \mu-\alpha c,
\end{equation*}
which is equivalent to \eqref{eq:NSCH4}.

Finally, \eqref{eq:ConvCVWeak} follows easily by passing to the limit in \eqref{eq:WeakApprox2} and 
using the fact that
\begin{equation*}
  \Psi_\eps \ve\to_{\eps \to 0} \ve \qquad \text{in}\ L^2(\Omega\times (0,T))^d
\end{equation*}
for all $\ve \in L^2(0,T,L^2_\sigma(\Omega))$.
Furthermore $\vc{n}\cdot\nabla c|_{\partial\Omega}=\lim_{j\to\infty}\vc{n}\cdot\nabla c_j|_{\partial\Omega}=0$ by the continuity of the trace operator.
\end{proof}

\section{Weak Solution}\label{sec:Main}

To construct a weak solution of the above system,
we use solutions of the approximate system \eqref{eq:NSCH1'}-\eqref{eq:NSCH4'} together with \eqref{eq:NSCH5}-\eqref{eq:NSCH6}.
The existence of weak solutions of the approximate system follows from Theorem~\ref{thm:ExistenceApprox}. In the following the solutions of the latter system are denoted by $(\ve_\eps,c_\eps,\mu_\eps)$ for $\eps>0$. Using the a~priori estimates given by \eqref{eq:EnergyEstimCH} and \eqref{eq:EstimCH}, we can conclude for a suitable subsequence $\eps_{i}\to_{i\to \infty} 0$ that
\begin{alignat}{2}\nonumber
\De \ve_{\eps_i} &\rightarrow \De \ve~~&&\text{weakly in}~~\mathit{L^q(Q_T)},
\\\nonumber
\ve_{\eps_i} &\rightarrow \ve~~&&\text{weakly in}~~\mathit{L^{q\frac{n+2}{n}}(Q_T)},
\\\nonumber
\Se \left(c_{\eps_i}, \De \ve_{\eps_i} \right) &\rightarrow \widetilde{\Se}~~&&\text{weakly in}~~\mathit{L^{q'}(Q_T)},
\\\label{weak-convergence}
\ve_{\eps_i}\otimes \ve_{\eps_i} \Phi_{\eps_i}\left( \ve_{\eps_i} \right) &\rightarrow \widetilde{\He}~~&&\text{weakly in}~~\mathit{L^{q\frac{n+2}{2n}}(Q_T)}.
\end{alignat}
Moreover, because of \eqref{eq:EstimCH}, \eqref{eq:NSCH3'}, and the Lemma of Aubin-Lions, it is easy to prove that
\begin{equation*}
   c_{\eps_i}\to_{i\to \infty}  c \qquad \text{in}\ L^2(0,T;C^1(\overline{\Omega}))
\end{equation*}
since $W^2_6(\Omega)\hookrightarrow C^1(\overline{\Omega})$ compactly. Interpolation with the boundedness of $c_\eps \in L^\infty(0,T;H^1(\Omega))$ yields
\begin{equation}\label{eq:StrongNablaC}
  \nabla c_{\eps_i}\to_{i\to \infty} \nabla c \qquad \text{in}\ L^4(Q_T).
\end{equation}


Let $ \Ke_{\eps}\in L^2(Q_T)^{d\times d}$ be such that
\begin{alignat}{1}\nonumber
\int_{Q_T}\Ke_{\eps}:\mathbf{D}\boldsymbol{\varphi} \sd (x,t) &= \int_{Q_T} \nabla c_{\eps} \otimes \nabla c_{\eps}: \mathbf{D} \Psi_{\eps} (\varphie)~d(x, t)\\
&- \int_{Q_T} \nabla c\otimes \nabla c : \mathbf{D} \varphie~d(x, t) 
\end{alignat}
for all $\boldsymbol{\varphi} \in L^2(0,T;H^1_0(\Omega)^d)$ and that
$\Ke_{\eps}\in L^2(Q_T)^d$ and
\begin{equation*}
  \|\Ke_\eps \|_{L^2(Q_T)}\leq C\|\Div (\nabla c\otimes\nabla c)-\Psi_\eps\Div
(\nabla c_\eps\otimes \nabla c_\eps)\|_{L^2(0,T;H^{-1}_0)}
\end{equation*}
for some $C>0$.
We can assume that $\Ke_{\eps}$ is pointwise a symmetric matrix.  Then
\begin{equation*}
\Ke_{\eps_i} \rightarrow 0~~\text{strongly in}~~\mathit{L^{2}(Q_T)^{d\times d}},
\end{equation*}
due to \eqref{eq:StrongNablaC}.

Since $q> \frac{2d}{d+2}$, there exists some $\sigma_0>1$ such that
$q\frac{d+2}{2d} > \sigma_0 >1$. Hence, due to~\eqref{weak-convergence}
we have for some $ \eps_{i} \rightarrow_{i\to \infty} 0$,
\begin{alignat}{2}\label{eq:StrongVConv}
  \ve_{\eps_i} &\rightarrow \ve &~~&\mathrm{strongly~~in}~~L^{2
    \sigma_0} (Q_T)
  \\
  \mathrm{and} ~~\ve_{\eps_i} \otimes \ve_{\eps_i} \Phi_{\eps_i}
  (|\ve_{\eps}|) &\rightarrow \ve \otimes \ve
  &&\mathrm{strongly~~in}~~L^{\sigma_0}(Q_T).
\end{alignat} 
We also have for $ i \rightarrow \infty, $ 
\begin{alignat}{2}
  \ve_{\eps_i} \rightarrow \ve~~~~\mathrm{strongly~~in}~~L^r (0, T;
  L^2 (\Omega)),\ \text{for all}\ 1\leq r < \infty
\end{alignat}
by interpolation of \eqref{eq:StrongVConv} with the boundedness of $(\ve_\eps)_{\eps\in (0,1)}\in L^\infty(0,T;L^2(\Omega))$.

Taking the limit of the weak form of the approximate system along the subsequence $ \eps_{i}, $
we obtain the following limit equation:
\begin{alignat}{1}
  \label{eq:limit_eq_v}
  - \int_{Q_T} \ve \cdot \partial_t \boldsymbol{\varphie} ~d(x, t) +
  \int_{Q_T} (\widetilde{\Se} - \ve \otimes \ve) : \mathbf{D}
  \boldsymbol{\varphie}~d(x, t) \\\nonumber =\int_{Q_T} \nabla c\otimes \nabla c : \De
  \varphie ~d(x, t) + \int_{\Omega} \ve_0 \cdot \boldsymbol{\varphie}
  (0)~dx.
\end{alignat}
for all $\boldsymbol{\varphie}\in C^\infty_0(Q_T)^d$ with $\Div \boldsymbol{\varphie}=0$.

By subtracting the above equation from the weak form of the approximate equations,
we have the following.
\begin{align*}
  \lefteqn{- \int_{Q_T} (\ve_{\eps} - \ve) \cdot \partial_t
    \boldsymbol{\varphi} ~d(x, t) + \int_{Q_T}
    \left(\mathbf{S}(c_{\eps}, \mathbf{D} \ve_{\eps}) -
      \widetilde{\mathbf{S}}\right) : \mathbf{D} \boldsymbol{\varphi}
    ~d(x, t)}
  \\
  \nonumber &=& \int_{Q_T} \left( \ve_{\eps} \otimes \ve_{\eps}
    \Phi_{\eps} ( \ve_{\eps} ) - \ve \otimes \ve \right) : \mathbf{D}
  \boldsymbol{\varphi} ~d(x, t) + \int_{Q_T}\Ke_{\eps}:\mathbf{D}\boldsymbol{\varphi}\sd (x,t).
\end{align*} 
Define $\ue_{\eps} := \ve_{\eps}-\ve$, then we can write this as
\begin{align}
  \label{eq:limit}
  \int\limits_{Q_T} \bfu_{\eps_i}
  \cdot\partial_t\varphie \,d(x,t) \,=\,
  \int\limits_{Q_T} \bfH_{\eps_i} : \nabla \varphie \,d(x,t)
\end{align}
for any $\varphie \in C^{\infty}(Q_T)^d $ with $ \Div
\varphie = 0 $ and $ {\rm supp} (\varphi) \subset \subset \Omega
\times [0, T)$, where
$\bfH_i:={\bfH}_{1,i}+\bfH_{2,i}$ with
\begin{align*}
  {\bfH}_{1, i}&:= {\bf S}(c_{\eps_i}, \mathbf{D} \ve_{\eps_i})-{\bf
    \tilde{S}},
  \\
  \bfH_{2, i}&:=\bfv_{\eps_i}\otimes \bfv_{\eps_i}
  \Phi_{\eps_i}(\ve_{\eps_i}) - \bfv\otimes \bfv - \Ke_{\eps_i} 
\end{align*}
Then we have the following convergences for suitable $\eps_i \to_{i \to
  \infty} 0$
\begin{alignat}{2}
  \ue_{\eps_i}  &\to 0 &~~&\text{weakly in } L^q(0,t; V_p(\Omega)),
  \\
  \ue_{\eps_i}  &\to 0 &~~&\text{strongly in } L^{2\sigma_0}(Q_T),
  \\
  \ue_{\eps_i}  &\to 0 &~~&\text{$*$-weakly in }
  L^\infty(0,T;L^2(\Omega)),
  \\
  \He_{1,i} &\to 0 &~~&\text{weakly in } L^{q'}(Q_T),
  \\
  \He_{1,2} &\to 0 &~~&\text{strongly in } L^{\sigma_0}(Q_T)
\end{alignat}
for some $1<\sigma_0<\min(q,q')$.
Let $I_0$ be a time interval, $B_0 \subset \R^d$ be a ball such that
$Q_0 := I_0 \times B_0 \subset\subset Q_T$. Let $\zeta\in
C^\infty_0(\frac 16 \Qz)$ with $\chi_{\frac 18 \Qz}\leq\zeta\leq
\chi_{\frac 16 \Qz}$. Then we can apply Theorem~\ref{thm:sol_lip} with
$\Ke = {\bf \tilde{S}} - {\bf S}(c, \mathbf{D} \ve)$ to obtain
\begin{align*}
  \limsup_{i\rightarrow\infty} \bigg|\int \big((\bfH_{1,i}+ {\bf
    \tilde{S}} - \bfS(c, \mathbf{D}\ve)):\nabla (\bfv_{\eps_i} - \bfv)
  \big) \zeta \chi_{\mathcal{O}_{i,k}^\complement} \,d(x,t)\bigg| \leq
  c\, 2^{-k/q}.
\end{align*}
In other words
\begin{align*}
  \limsup_{i\rightarrow\infty} \bigg|\int \Big(\big(\bfS(c_{\eps_i},
  \mathbf{D}\bfv_{\eps_i})-\bfS(c, \mathbf{D}\bfv)\big): \mathbf{D}(\ve_{\eps_i}
  - \bfv) \Big) \zeta \chi_{\mathcal{O}_{i,k}^\complement}
  \,d(x,t)\bigg| \leq c\, 2^{-k/q}.
\end{align*}
Let $\theta \in (0,1)$. Then by H{\"o}lder's inequality and $0 \leq \zeta
\leq 1$ we get
\begin{align*}
  \lefteqn{\limsup_{i\rightarrow\infty} \bigg|\int
    \Big(\big(\bfS(c_{\eps_i}, \mathbf{D}\bfv_{\eps_i})-\bfS(c,
    \mathbf{D}\bfv)\big): \mathbf{D}(\ve_{\eps_i} - \bfv) \Big)^\theta
    \zeta \chi_{\mathcal{O}_{i,k}} \,d(x,t)\bigg|} \qquad & 
  \qquad\qquad\qquad\qquad\qquad\qquad\qquad\qquad\qquad\qquad\qquad\qquad
  \\
  &\leq c\, \limsup_{i \to \infty} \abs{\mathcal{O}_{i,k}}^{1-\theta}
  \leq c\, 2^{-(1-\theta)\frac kq}.
\end{align*}
This, the previous estimate and H{\"o}lder's inequality give
\begin{align*}
  \limsup_{i\rightarrow\infty} \bigg|\int \Big(\big(\bfS(c_{\eps_i},
  \mathbf{D}\bfv_{\eps_i})-\bfS(c, \mathbf{D}\bfv)\big):
  \mathbf{D}(\ve_{\eps_i} - \bfv) \Big)^\theta \zeta \,d(x,t)\bigg|
  &\leq c\, 2^{-(1-\theta)\frac kq}.
\end{align*}
For $k\to \infty$ the right hand side converges to zero.  Now, the
monotonicity of $\bfS$ and $\zeta \geq \chi_{\frac 18 Q_0}$ implies
that $\bfS(c_{\eps_i}, \mathbf{D}\bfv_{\eps_i})\to\bfS(c,
\mathbf{D}\bfv)$ a.e. on $\frac 18 \Qz$. Since~$Q_0$ was an arbitrary
space time cylinder in $Q_T$, we get $\bfS(c_{\eps_i},
\mathbf{D}\bfv_{\eps_i})\to\bfS(c, \mathbf{D}\bfv)$ a.e. on $Q_T$.
Since $\bfS(c_{\eps_i}, \bfD \bfv_{\eps_i}) \to \bfS(c, \bfD \bfv)$
weakly in $L^{q'}(Q_T)$, we get as desired $\tilde{\bfS} = \bfS(c,
\bfD \bfv)$. This and~\eqref{eq:limit_eq_v} prove that $\bfv$ solves
\begin{alignat}{1}
  \label{eq:limit_eq_v_final}
  - \int_{Q_T} \ve \cdot \partial_t \boldsymbol{\varphie} ~d(x, t) +
  \int_{Q_T} (\bfS(c,\bfD \bfv) - \ve \otimes \ve) : \mathbf{D}
  \boldsymbol{\varphie}~d(x, t) \\\nonumber =\int_{Q_T} {\Ke} : \De
  \varphie ~d(x, t) + \int_{\Omega} \ve_0 \cdot \boldsymbol{\varphie}
  (0)~dx
\end{alignat}
for any $\varphie \in C^{\infty}(Q_T)^d $ with $ \Div
\varphie = 0 $ and $ {\rm supp} (\varphie) \subset \subset \Omega
\times [0, T)$.

\subsection*{Acknowledgement} 
This work was supported by the SPP 1506 "Transport Processes
at Fluidic Interfaces" of the German Science Foundation (DFG) through
the grant AB 285/4-1. Moreover, T. was supported by JSPS Research Fellowships for Young 
Scientists
and by JSPS Program for Leading Graduate Schools. The supports are gratefully acknowledged.


\def\cprime{$'$} \def\ocirc#1{\ifmmode\setbox0=\hbox{$#1$}\dimen0=\ht0
  \advance\dimen0 by1pt\rlap{\hbox to\wd0{\hss\raise\dimen0
  \hbox{\hskip.2em$\scriptscriptstyle\circ$}\hss}}#1\else {\accent"17 #1}\fi}

\end{document}